\documentclass[reqno,12pt]{amsart}
\usepackage{amsmath,amsthm,amssymb,amsfonts,epsfig,color,graphicx,enumerate,accents,subfigure}
\usepackage[a4paper,margin=2.98truecm]{geometry}

\def\eps{{\varepsilon}}

\def\R{\mathbb{R}}
\def\O{\Omega}

\def\A{\mathcal{A}}

\def\F{\mathcal{F}}
\def\HH{\mathcal{H}}

\def\PP{\mathcal{P}}

\newcommand{\be}{\begin{equation}}
\newcommand{\ee}{\end{equation}}
\newcommand{\bib}[4]{\bibitem{#1}{\sc#2: }{\it#3. }{#4.}}
\newcommand{\cp}{\mathop{\rm cap}\nolimits}

\numberwithin{equation}{section}
\theoremstyle{plain}
\newtheorem{theo}{Theorem}[section]

\newtheorem{prop}[theo]{Proposition}

\theoremstyle{remark}
\newtheorem{rema}[theo]{Remark}
\newtheorem{exam}[theo]{Example}

\newenvironment{ack}{{\bf Acknowledgements. }}

\title[Optimal shapes for general integral functionals]{Optimal shapes for general integral functionals}

\author{Giuseppe Buttazzo}
\address{Dipartimento di Matematica, Universit\`a di Pisa, Largo B. Pontecorvo 5, 56126 Pisa, ITALY}
\email{giuseppe.buttazzo@unipi.it}

\author{Harish Shrivastava}
\address{Dipartimento di Matematica, Universit\`a di Pisa, Largo B. Pontecorvo 5, 56126 Pisa, ITALY}
\email{harish.niser@gmail.com}

\begin{document}
\maketitle

\begin{abstract}
We consider shape optimization problems for general integral functionals of the calculus of variations, defined on a domain $\O$ that varies over all subdomains of a given bounded domain $D$ of $\R^d$. We show in a rather elementary way the existence of a solution that is in general a quasi open set. Under very mild conditions we show that the optimal domain is actually open and with finite perimeter. Some counterexamples show that in general this does not occur.
\end{abstract}

\medskip
\textbf{Keywords:} shape optimization, quasi open sets, finite perimeter, integral functionals.

\textbf{2010 Mathematics Subject Classification:} 49Q10, 49A15, 49A50, 35J20, 35D10

\section{Introduction}\label{sintr}

In this paper we consider a shape optimization problem for a general integral functional of the form
\be\label{intfun}
F(u,\O)=\int_\O f(x,u,\nabla u)\,dx.
\ee
Setting
$$\F(\O)=\min\big\{F(u,\O)\ :\ u\in W^{1,p}_0(\O)\big\}$$
the problem we are dealing with is written as
\be\label{pbintro}
\min\big\{\F(\O)\ :\ \O\subset D\big\}.
\ee
Here $p>1$ is a fixed real number, $D$ is a given bounded domain of $\R^d$, and $f$ is a general integrand satisfying suitable rather mild assumptions. Note that in problem \eqref{pbintro} the volume constraint can be incorporated into the cost functional $\F$ by means of a Lagrange multiplier of the form $\lambda|\O|$ or more generally $\int_\O\lambda(x)\,dx$. For a detailed presentation of shape optimization problems we refer the interested reader to the books \cite{bubu05} and \cite{hepi05}. 

The first result is Theorem \ref{exth}, which gives the existence of an optimal domain $\O_{opt}$. This optimal domain belongs to the class of $p$-quasi open sets, defined as the sets $\{u>0\}$ for some function $u\in W^{1,p}_0(D)$. As a consequence, if $p>d$ these optimal sets are actually open, but if $p\le d$ this fact does not occur any more under the very general assumptions we made, see Example \ref{exnotopen}.

The existence of optimal sets $\O_{opt}$ could have been obtained through a generalization of a result in \cite{bdm93} to the case $p>1$, making use of a $\gamma_p$-convergence on the class of $p$-quasi open sets. However, we have preferred to give an independent proof that, in the particular case of integral functionals of the form \eqref{intfun}, is much simpler.

In order to obtain that the optimal sets $\O_{opt}$ are open, we need slightly stronger assumptions: this is the goal of Theorem \ref{thopen}, in which we use the H\"older continuity result of \cite{gigi84}, \cite{giu03} on the minimizers of general integral functionals.

Finally, in Theorem \ref{tper} we prove, under rather general assumptions on the integrand $f$, that $\O_{opt}$ has a finite perimeter. This result is obtained by adapting a previous result of \cite{db2012} to the general case of an integrand $f$ with a $p$-growth.

\section{Setting of the problem and existence result}\label{sexis}

We recall here some well-known notions from the Sobolev spaces theory; for all details we refer to \cite{bubu05} and to \cite{maz11}.

In all the paper $p>1$ will be a fixed real number. For every set $E\subset\R^d$ the $p$-capacity of $E$ is defined as
$$\cp_p(E)=\inf\Big\{\|u\|_{W^{1,p}(\R^d)}\ :\ u\in W^{1,p}(\R^d),\ u\ge1\hbox{ a.e. in a neighborhood of }E\Big\}.$$
We say that a property $\PP(x)$ holds {\it$p$-quasi everywhere} (shortly q.e.) in a set $E$ if the set of points of $x\in E$ for which $\PP(x)$ does not hold has $p$-capacity zero; the expression {\it almost everywhere} (shortly a.e.) refers, as usual, to the Lebesgue measure.

A set $\O$ is called {\it$p$-quasi open} if $\O=\{u>0\}$ for a suitable function $u\in W^{1,p}(\R^d)$. It has to be noticed that, since for $p>d$ the $W^{1,p}(\R^d)$ functions are H\"older continuous, we have that in this case $p$-quasi open sets are actually open. If $\O$ is $p$-quasi open we may define the Sobolev space $W^{1,p}_0(\O)$ as
$$W^{1,p}_0(\O)=\big\{u\in W^{1,p}(\R^d)\ :\ u=0\hbox{ $\cp_p$ q.e. on }\R^d\setminus\O\big\}.$$
We notice that this definition coincides with the usual one in the case when $\O$ is open.

In the following we fix a bounded domain $D$ of $\R^d$ and we consider the admissible class
$$\A=\big\{\O\subset D\ :\ \O\hbox{ $p$-quasi open}\big\}.$$
For every $\O\in\A$ and $u\in W^{1,p}_0(\O)$ we define the integral functional
\be\label{intfun}
F(u,\O)=\int_\O f(x,u,\nabla u)\,dx
\ee
where the integrand $f$ is assumed to verify the following conditions:
\begin{itemize}
\item[(f1)]$f(x,s,z)$ is measurable in $x$, lower semicontinuous in $(s,z)$, convex in $z$;
\item[(f2)]there exist $c>0$, $a\in L^1(D)$, and $\alpha<\lambda_{1,p}(D)$ such that
$$c\big(|z|^p-\alpha|s|^p-a(x)\big)\le f(x,s,z)\qquad\hbox{for every $x,s,z$},$$
being $\lambda_{1,p}(D)$ the first Dirichlet eigenvalue of the $p$-Laplacian on $D$, defined as
$$\lambda_{1,p}(D)=\min\left\{\int_D|\nabla u|^p\,dx\ :\ u\in W^{1,p}_0(D),\ \int_D|u|^p\,dx=1\right\}.$$
\item[(f3)]$f(x,0,0)\ge0$.
\end{itemize}
It is well-known (see for instance \cite{bu89}) that under conditions (f1) and (f2) for every $\O\in\A$ the functional $F(\cdot,\O)$ defined in \eqref{intfun} is lower semicontinuous with respect to the weak convergence in $W^{1,p}_0(\O)$ and that the minimum problem
\be\label{minf}
\min\left\{F(u,\O)\ :\ u\in W^{1,p}_0(\O)\right\}
\ee
admits a solution. Let us denote by $\F(\O)$ the minimum value in \eqref{minf}. The shape optimization problem we deal with is
\be\label{shoptp}
\min\big\{\F(\O)\ :\ \O\in\A\big\}
\ee

In the following theorem we prove that the shape optimization problem above admits a solution. For the proof we could use the general theory of $\gamma$-convergence and weak $\gamma$-convergence (see \cite{bubu05}), and the fact that the shape functional $\F$ is monotonically decreasing with respect to the set inclusion; however, in our case a simpler proof is available and we report this one.

\begin{theo}\label{exth}
Under assumptions (f1), (f2), (f3) the shape optimization problem \eqref{shoptp} admits a solution.
\end{theo}

\begin{proof}
Consider the auxiliary minimum problem
\be\label{auxpb}
\min\left\{\int_D f(x,u,\nabla u)1_{\{u\ne0\}}\,dx\ :\ u\in W^{1,p}_0(D)\right\}.
\ee
Since
\[\begin{split}
\int_D f(x,u,\nabla u)1_{\{u\ne0\}}\,dx
&=\int_D f(x,u,\nabla u)\,dx-\int_D f(x,0,0)1_{\{u=0\}}\,dx\\
&=\int_D\big[f(x,u,\nabla u)-f(x,0,0)\big]\,dx+\int_D f(x,0,0)1_{\{u\ne0\}}\,dx,
\end{split}\]
Problem \eqref{auxpb} can be rewritten as
$$\min\left\{\int_D\big[f(x,u,\nabla u)-f(x,0,0)+f(x,0,0)1_{\{u\ne0\}}\big]\,dx\ :\ u\in W^{1,p}_0(D)\right\}$$
and, thanks to assumptions (f1) and (f2), it verifies the lower semicontinuity and coercivity properties that guarantee it admits a solution $\bar u$. We claim that the $p$-quasi open set $\overline\O=\{\bar u\ne0\}$ solves the shape optimization problem \eqref{shoptp}. Indeed, let $\O\in\A$ and let $u_\O$ be the solution of the minimum problem \eqref{minf}; then we have
\[\begin{split}
\F(\O)&=\int_\O f(x,u_\O,\nabla u_\O)\,dx\\
&=\int_D f(x,u_\O,\nabla u_\O)\,dx-\int_{D\setminus\O}f(x,0,0)\,dx\\
&=\int_D f(x,u_\O,\nabla u_\O)1_{\{u_\O\ne0\}}\,dx+\int_D f(x,0,0)\big[1_{\{u_\O=0\}}-1_{D\setminus\O}\big]\,dx\\
&\ge\int_D f(x,\bar u,\nabla\bar u)1_{\{\bar u\ne0\}}\,dx+\int_\O f(x,0,0)1_{\{u_\O=0\}}\,dx\ge\F(\overline\O)
\end{split}\]
where the last inequality follows from the definition of $\overline\O$ and from assumption (f3).
\end{proof}

\section{Cases when optimal domains are open}\label{sopen}

In the present section we show that, under mild additional assumptions on the integrand $f$, the optimal domain $\overline\O$ of problem \eqref{shoptp}, obtained in Theorem \ref{exth} is actually an open set. To do this we show that the solution $\bar u$ of the auxiliary minimum problem \eqref{auxpb} is a continuous function. This follows by means of a well-known result of Giaquinta and Giusti in \cite{gigi84} (see also \cite{giu03}), that we summarize here below for the sake of completeness.

\begin{theo}\label{gigi}
Let $\bar u$ be a solution of the minimum problem
$$\min\left\{\int_D h(x,u,\nabla u)\,dx\ :\ u\in W^{1,p}_0(D)\right\}$$
where the integrand $h$ satisfies the condition
\be\label{gigicond}
c\big(|z|^p-b(x)|s|^\gamma-g(x)\big)\le h(x,s,z)\le C\big(|z|^p+b(x)|s|^\gamma+g(x)\big)
\ee
for all $x,s,z$, where $p>1$, $0<c\le C$, $p\le\gamma<p^*$, $b\in L^q_{loc}(D)$, $g\in L^\sigma_{loc}(D)$ being $p^*=dp/(d-p)$ ($p^*=+\infty$ if $p\ge d$) the Sobolev exponent relative to $p$, $\sigma>d/p$, $q>p^*/(p^*-\gamma)$. Then $\bar u$ is locally H\"older continuous in $D$.
\end{theo}

\begin{rema}
In the paper \cite{gigi84} the integrand $h$ above was assumed of Carath\'eodory type, but in fact condition (f1) is still enough, provided condition \eqref{gigicond} is satisfied. Actually, as the authors say, even the convexity of $h$ with respect to $z$ is not needed, if we assume that a solution $\bar u$ exists.
\end{rema}

\begin{rema}
Of course, the result above is nontrivial only in the case $p\le d$; indeed, if $p>d$ the H\"older continuity of $\bar u$ simply follows from the Sobolev embedding theorem.
\end{rema}

We can now apply Theorem \ref{gigi} to obtain that in a large number of situations the optimal set $\overline\O$ obtained in Theorem \ref{exth} is actually an open set.

\begin{theo}\label{thopen}
Assume that the integrand $f$ satisfies conditions (f1), (f2), (f3), \eqref{gigicond}
Then the optimal domain $\overline\O$ of problem \eqref{shoptp}, obtained in Theorem \ref{exth} is an open set.
\end{theo}

\begin{proof}
Since $\overline\O=\{\bar u\ne0\}$ where $\bar u$ is a solution of the auxiliary problem \eqref{auxpb}, it is enough to show that the function $\bar u$ is continuous on $D$. We have for every $u\in W^{1,p}_0(D)$
$$\int_D f(x,u,\nabla u)1_{\{u\ne0\}}\,dx=\int_D \big[f(x,u,\nabla u)-f(x,0,0)1_{\{u=0\}}\big]\,dx$$
and the integrand
$$h(x,s,z)=f(x,s,z)-f(x,0,0)1_{\{s=0\}}$$
satisfies the conditions of Theorem \ref{gigi}. Then the H\"older continuity of $\bar u$ follows.
\end{proof}

\begin{rema}
In general, under the sole existence assumptions (f1), (f2), (f3), we do not expect that the optimal domain $\overline\O$ be open. In \cite{bhp05} the authors consider the particular case
$$F(u,\O)=\frac12\int_\O|\nabla u|^2\,dx-\langle h,u\rangle$$
under a volume constraint of the form $|\O|=m$, and refer to \cite{hay99} for a counterexample to the fact that the solution $\overline\O$ is open, when the function $h$ is in $H^{-1}(D)$. In the following section we show that a counterexample can be constructed even in the case $h\in H^{-1}(D)\cap L^1(D)$.
\end{rema}

\section{Optimal domains that are not open}\label{snotopen}

As we have seen in Theorem \ref{thopen} quite mild assumptions on the integrand $f$ imply that the optimal domains $\O_{opt}$ are open sets. In this section we show that, when these assumptions are not satisfied, the optimal domains may be not better than quasi open sets, even in very simple cases as the Dirichlet energy
$$F(u,\O)=\int_\O\Big[\frac1p|\nabla u|^p-f(x)u\Big]\,dx.$$

We start by a preliminary result.

\begin{prop}\label{noaux}
Let $f$ be an integrand satisfying conditions (f1) and (f2), and assume that $f(x,0,0)=0$. Let $\bar u$ be a solution of the minimum problem
$$\min\left\{\int_D f(x,u,\nabla u)\,dx\ :\ u\in W^{1,p}_0(D)\right\}$$
and denote by $\overline\O$ the $p$-quasi open set $\{\bar u\ne0\}$. Then $\overline\O$ is a solution of the shape optimization problem \eqref{pbintro}.
\end{prop}

\begin{proof}
Setting
$$\F(\O)=\min\left\{\int_\O f(x,u,\nabla u)\,dx\ :\ u\in W^{1,p}_0(\O)\right\}$$
we have to show that for every $p$-quasi open set $\O\subset D$ we have
\be\label{ovopt}
\F(\O)\ge\F(\overline\O).
\ee
Using the optimality of $\bar u$ and the fact that $f(x,0,0)=0$ we obtain
\[\begin{split}
\min\left\{\int_\O f(x,u,\nabla u)\,dx\ :\ u\in W^{1,p}_0(\O)\right\}
&=\min\left\{\int_D f(x,u,\nabla u)\,dx\ :\ u\in W^{1,p}_0(\O)\right\}\\
&\ge\min\left\{\int_D f(x,u,\nabla u)\,dx\ :\ u\in W^{1,p}_0(D)\right\}\\
&=\int_D f(x,\bar u,\nabla\bar u)\,dx=\int_{\overline\O}f(x,\bar u,\nabla\bar u)\,dx\\
&\ge\min\left\{\int_{\overline\O}f(x,u,\nabla u)\,dx\ :\ u\in W^{1,p}_0(\overline\O)\right\},
\end{split}\]
which implies \eqref{ovopt}.
\end{proof}

Consider now the shape optimization problem for the Dirichlet energy:
\be\label{dircounter}
\min\Big\{\F(\O)\ :\ \O\subset D\Big\}
\ee
where
\be\label{enecount}
\F(\O)=\min\left\{\int_\O\Big[\frac1p|\nabla u|^p-f(x)u\Big]\,dx\ :\ u\in W^{1,p}_0(\O)\right\}.
\ee
Here $f$ is always assumed in $W^{-1,p'}(D)\cap L^1(D)$. This formulation incorporates in many cases the same optimization problem as \eqref{dircounter} with a measure constraint of the form $\big\{|\O|\le m\big\}$, through the addition of a term $\lambda|\O|$ in the cost functional, where $\lambda$ is a Lagrange multiplier. However, a rigorous proof of the equivalence of the two formulations is not available in full generality, see for instance \cite{bhp05} and \cite{hay99} for a discussion on this matter. By Theorem \ref{thopen}, when in addition $f\in L^q(D)$ with $q>d/p$, we obtain that the optimal domains $\O_{opt}$ are open sets. In the following example we show that $\O_{opt}$ may be not open if $p\le d$ and $q=1$. More precisely, we show that for every $p$-quasi open set $\O$ we can find $f\in W^{-1,p'}(D)\cap L^1(D)$ such that $\O$ is the optimal domain for the shape optimization problem \eqref{dircounter}.

\begin{exam}\label{exnotopen}
Let $p>1$ and let $\overline\O$ be any quasi open subset of $D$. Let $w$ be the torsion function associated to $\overline\O$, that is the unique solution of the PDE
$$-\Delta_p w=1\hbox{ in }\overline\O,\qquad w\in W^{1,p}_0(\overline\O),$$
intended as the minimizer on $W^{1,p}_0(\overline\O)$ of the functional
$$\int_\O\Big[\frac1p|\nabla u|^p-u\Big]\,dx$$
or equivalently as the solution of the PDE in its weak form
$$\int_D|\nabla w|^{p-2}\nabla w\nabla\phi\,dx=\int_D w\phi\,dx\qquad\hbox{for every }\phi\in W^{1,p}_0(\overline\O).$$
We claim that the function $f=-\Delta_p(w^{p'})$ is in $W^{-1,p'}(D)\cap L^1(D)$. Indeed, by the maximum principle $w$ is bounded and, since $\nabla(w^{p'})=p'w^{p'-1}\nabla w$, we get that $w^{p'}\in W^{1,p}_0(D)$ and so $f\in W^{-1,p'}(D)$. Moreover, the equality
$$\Delta_p(w^{p'})=\Big(\frac{p}{p-1}\Big)^{p-1}\Big[w\Delta_p w+|\nabla w|^p\Big]$$
gives that $f\in L^1(D)$. Indeed, $|\nabla w|^p\in L^1(D)$ and, by the equality $w\Delta_p w=-w$ on $D$ we obtain that the term $w\Delta_p w$ is actually in $L^\infty(D)$.

Now, we apply Proposition \ref{noaux} with $f$ as above; this implies that the function $\bar u$ coincides with $w^{p'}$ and so $\O_{opt}$ is the set $\{w^{p'}\ne0\}$, which coincides with $\overline\O$.
\end{exam}

\begin{rema}
We have seen that if $p>d$ or if the function $f$ in \eqref{enecount} is in $L^q(D)$ with $q>d/p$ then the optimal set $\O_{opt}$ is open. On the contrary, if $q=1$ we can construct a counterexample showing that $\O_{opt}$ is merely a $p$-quasi open set. This picture is sharp when $p=d$ in the sense that in this case $q=1$ is the borderline situation and $q\le1$ gives a counterexample, while $q>1$ gives that $\O_{opt}$ is open. When $p<d$ we do not know if similar counterexamples hold in the case $1<q\le d/p$.
\end{rema}

\section{Cases when optimal domains have finite perimeter}\label{speri}

In this section we show that, under some assumptions slightly stronger than (f1), (f2), f3) the optimal set $\overline\O$ obtained in Section \ref{sexis} has a finite perimeter. We adapt the proof contained in \cite{db2012} to our general case. The assumptions we need are:
\begin{itemize}
\item[(f2'')]there exist $c>0$ and $\alpha<\lambda_{1,p}(D)$ such that for every $x,s,z$
$$c\big(|z|^p-\alpha|s|^p+1\big)\le f(x,s,z);$$
\item[(f3'')]there exist $K>0$ and $a\in L^1(D)$ such that for every $x,s,t,z$
$$\big|f(x,s,z)-f(x,t,z)\big|\le K|s-t|\big(a(x)+|s|^{p^*}+|t|^{p^*}+|z|^p\big),$$
where $p^*=dp/(d-p)$ (with $p^*$ any positive number if $p=d$ and $|\cdot|^{p^*}$ replaced by any continuous function if $p>d$) is the Sobolev exponent associated to $p$.
\end{itemize}
Note that in particular condition (f2'') implies a condition stronger than (f3):
$$f(x,0,0)\ge c\qquad\hbox{with }c>0.$$

\begin{theo}\label{tper}
Under assumptions (f1), (f2''), (f3''), the optimal domain $\overline\O$ obtained in the existence Theorem \eqref{exth} has a finite perimeter.
\end{theo}

\begin{proof}
Let $u$ be a solution of the auxiliary minimization problem \eqref{auxpb} and let, for every $\eps>0$
$$u_\eps=(u-\eps)^+-(u-\eps)^-,\qquad A_\eps=\big\{0<|u|\le\eps\big\}.$$
Note that
$$u_\eps=\begin{cases}
u-\eps&\hbox{if }u>\eps\\
u+\eps&\hbox{if }u<-\eps\\
0&\hbox{if }|u|\le\eps
\end{cases}
\qquad\hbox{and}\qquad
\nabla u_\eps=\begin{cases}
\nabla u&\hbox{a.e. on }\{|u|>\eps\}\\
0&\hbox{a.e. on }\{|u|\le\eps\}.
\end{cases}$$
Then, by the optimality of the function $u$, we have
\[\begin{split}
\int_D f(x,u,\nabla u)1_{\{u\ne0\}}\,dx
&\le\int_D f(x,u_\eps,\nabla u_\eps)1_{\{u_\eps\ne0\}}\,dx\\
&=\int_{D\setminus A_\eps}f(x,u_\eps,\nabla u)\,dx
\end{split}\]
so that, using assumption (f3''),
\[\begin{split}
\int_{A_\eps}f(x,u,\nabla u)\,dx
&\le\int_{D\setminus A_\eps}\big|f(x,u_\eps,\nabla u)-f(x,u,\nabla u)\big|\,dx\\
&\le C\int_D|u_\eps-u|\big(a(x)+|u_\eps|^{p^*}+|u|^{p^*}+|\nabla u|^p\big)\,dx\le C\eps.
\end{split}\]
We now use assumption (f2'') and we obtain
$$C\eps\ge\int_{A_\eps}f(x,u,\nabla u)\,dx\ge c\int_{A_\eps}\big(|\nabla u|^p-\alpha|u|^p\big)\,dx+c|A_\eps|,$$
which implies
$$\int_{A_\eps}|\nabla u|^p\,dx+|A_\eps|\le C\eps.$$
By H\"older inequality this gives
$$\int_{A_\eps}|\nabla u|\,dx\le C\eps.$$
We use now the coarea formula and we deduce
$$\int_0^\eps\HH^{d-1}\big(\partial^*\{|u|>t\}\big)\,dt\le C\eps.$$
Thus there exists a sequence $\delta_n\to0$ such that
$$\HH^{d-1}\big(\partial^*\{|u|>\delta_n\}\big)\le C\qquad\hbox{for every }n$$
and finally this implies that
$$\HH^{d-1}(\partial^*\overline\O)=\HH^{d-1}\big(\partial^*\{|u|>0\}\big)\le C,$$
as required.
\end{proof}

\ack This work is part of the project 2015PA5MP7 {\it``Calcolo delle Variazioni''} funded by the Italian Ministry of Research and University. The first author is members of the {\it``Gruppo Nazionale per l'Analisi Matematica, la Probabilit� e le loro Applicazioni''} (GNAMPA) of the {\it``Istituto Nazionale di Alta Matematica''} (INDAM). The second author gratefully acknowledges the financial support of the Doctoral School in Mathematics of the University of Pisa.


\end{document}